\documentclass[12pt]{amsart}
\usepackage[utf8]{inputenc}
\usepackage{amsthm,stmaryrd} 
\usepackage{amssymb,enumerate}
\usepackage{epsfig}
\usepackage[usenames,dvipsnames]{color}
\usepackage{verbatim}
\usepackage[bookmarksopen=true,bookmarksnumbered=true,bookmarksopenlevel=4]{hyperref}
\usepackage{mathrsfs}
\usepackage{dsfont}
\usepackage[all]{xy}
\usepackage{tikz}
\usetikzlibrary{matrix,arrows,decorations.pathmorphing}

\newcommand{\Forgetful}{{\operatorname{Forgetful}}}
\newcommand{\TV}{{\operatorname{TV}_\cat}}
\renewcommand{\Vec}{{{\mathcal V}\!\textrm{ect}_\kk}}
\newcommand{\Kuni}{\mathcal K}
\newcommand{\Kup}{\operatorname{Ku}}

\newcommand{\IP}{{\mathbf O_{\Po}}}

\newcommand{\Hb}{{\mathscr H}_b}
\newcommand{\Po}{{{\mathbb P}}}
\newcommand{\dd}{{\wt d}}

\newcommand{\tcoev}{\stackrel{\longrightarrow}{\operatorname{coev}}}
\newcommand{\tev}{\stackrel{\longrightarrow}{\operatorname{ev}}}
\newcommand{\ev}{\stackrel{\longleftarrow}{\operatorname{ev}}}
\newcommand{\coev}{\stackrel{\longleftarrow}{\operatorname{coev}}}
\newcommand{\rcoev}{\stackrel{\longrightarrow}{\operatorname{coev}}}
\newcommand{\rev}{\stackrel{\longrightarrow}{\operatorname{ev}}}
\newcommand{\lev}{\stackrel{\longleftarrow}{\operatorname{ev}}}
\newcommand{\lcoev}{\stackrel{\longleftarrow}{\operatorname{coev}}}

\newcommand{\brk}[1]{{\left\langle{#1}\right\rangle}}

\newcommand{\FK}{{\Bbbk}}
\newcommand{\qt}{\operatorname{\mathsf{t}}}
\newcommand{\Int}{\operatorname{Int}}

\newcommand{\unit}{\mathds{1}}
\newcommand{\cat}{\mathscr{C}}

\newcommand{\Id}{\operatorname{Id}}
\newcommand{\bp}[1]{{\left(#1\right)}}

\newcommand{\End}{\operatorname{End}}
\newcommand{\Hom}{\operatorname{Hom}}

\newcommand{\R}{\ensuremath{\mathbb{R}} }

\newcommand{\wt}{\widetilde}

\newcommand{\ms}[1]{\mbox{\tiny$#1$}}

\newcommand{\I}{\mathcal I}
\newcommand{\ideal}{\I}
\newcommand{\Proj}{{{\operatorname{\mathsf{Proj}}}}}

\newcommand{\kk}{\Bbbk}
\newcommand{\mt}{\mathsf t}
\newcommand{\qdim}{\operatorname{qdim}}
\renewcommand{\SS}{\Sigma}
\newcommand{\Rib}{\operatorname{Rib}}
\newcommand{\LL}{\mathcal{L}}
\newcommand{\Hp}{\mathscr{H}_{\I}}
\newcommand{\Ho}{{\ensuremath{A}}}
\newcommand{\G}{{\Gamma}}
\newcommand{\cut}{{\operatorname{cut}}}

\newcommand{\pathFig}{}
\newcommand{\epsh}[2]
         {\begin{array}{c} \hspace{-1.3mm}
        \raisebox{-4pt}{\epsfig{figure=\pathFig#1,height=#2}}
        \hspace{-1.9mm}\end{array}}

\newtheorem{theorem}{Theorem}[section]
\newtheorem{proposition}[theorem]{Proposition}
\newtheorem{lemma}[theorem]{Lemma}

\newcounter{IntroCounter}
\stepcounter{IntroCounter}

\theoremstyle{definition}

\newtheorem{definition}[theorem]{Definition}

\theoremstyle{remark}
\newtheorem{remark}[theorem]{Remark}
\newtheorem{claim}[theorem]{Claim}

\newcounter{exo} \newcounter{numexercice}
\renewcommand{\theexo}{\arabic{exo}}

\begin{document}
\title[Invariants in Unimodular Categories]{Kuperberg and Turaev-Viro Invariants in Unimodular Categories}

\author[F. Costantino]{Francesco Costantino}
\address{Institut de Math\'ematiques de Toulouse\\
118 route de Narbonne\\
 Toulouse F-31062}
\email{francesco.costantino@math.univ-toulouse.fr}

\author[N. Geer]{Nathan Geer}
\address{Mathematics \& Statistics\\
  Utah State University \\
  Logan, Utah 84322, USA}
\thanks{This work is supported by the NSF FRG Collaborative Research Grant DMS-1664387.  }\
\email{nathan.geer@gmail.com}

\author[B. Patureau-Mirand]{Bertrand Patureau-Mirand}
\address{UMR 6205, LMBA, universit\'e de Bretagne-Sud, universit\'e
  europ\'eenne de Bretagne, BP 573, 56017 Vannes, France }
\email{bertrand.patureau@univ-ubs.fr}

\author[V. Turaev]{Vladimir Turaev}
\address{Department of Mathematics \\
  Indiana University \\
  Rawles Hall, 831 East 3rd st \\
  Bloomington, IN 47405, USA}
\email{vtouraev@indiana.edu}

\begin{abstract}
  We give a categorical setting in which Penrose graphical calculus
  naturally extends to graphs drawn on the boundary of a handlebody.  We use it to
  introduce invariants of 3-manifolds presented by Heegaard
  splittings. 
  We recover Kuperberg invariants when the category arises
  from an involutory Hopf algebra and Turaev-Viro invariants when the
  category is semi-simple and spherical.
\end{abstract}

\maketitle
\setcounter{tocdepth}{1}
\tableofcontents
\setcounter{tocdepth}{3}

\section{Introduction}
A remarkable achievement of  the low-dimensional topology in the last 30 years was a discovery of   deep relations between   topology  and the theory of monoidal (tensor) categories. This development was initiated by V. Jones' introduction of his famous knot polynomial;  by now it englobes many aspects of low-dimensional topology including 3-manifold invariants, representations of mapping class groups of surfaces, topological quantum field theories in dimensions 2 and 3, etc.
 In particular, it was shown that monoidal categories satisfying certain conditions and carrying appropriate additional structures give rise to topological invariants of 3-dimensional manifolds, see \cite{Tu, TuV}.  This  has instigated extensive research in the theory of monoidal categories aiming at construction (and eventually classification) of monoidal categories with required properties.  
At the same time, this development has inspired  a number of parallel  approaches not involving  monoidal categories but  using related algebraic objects. One such approach  is due to G. Kuperberg \cite{Ku90} who derived invariants of 3-manifolds from  involutory Hopf algebras. The initial aim of this paper was to recover Kuperberg's invariants in terms of monoidal categories. To this end  we introduce    a  new construction of 3-manifold invariants from monoidal categories. We show that our method produces both the Kuperberg invariants and  the standard Turaev-Viro invariants.  Other generalizations of Kuperberg invariants   were considered by Kashaev and Virelizier in \cite{KV}. 

The first main result of our paper is a construction -  in a general categorical setting - of  an invariant of graphs on the boundary of a handlebody.
With some additional categorical structure we extend this invariant to bichrome graphs  which  split as the union of two subgraphs, the
blue and the red. 
Here the red part has the following key property: any edge of the graph can be slid over a red curve.  Our definition of a 3-manifold invariant  uses Heegaard decompositions of 3-manifolds 
as unions of two handlebodies glued along their common boundary.
Such a decomposition is 
encoded via
a red-colored set of disjoint simple closed curves on the boundary of one  handlebody forming  a complete set of meridians for the  second handlebody.

On the algebraic side, we use two main tools - the modified traces
(m-traces) and the chromatic morphisms. The m-traces generalize the
usual trace of endomorphisms of objects of a monoidal category to
situations where the standard trace
vanishes. The m-traces first
appeared in \cite{GKP1,GKP2, GPV} and have been successfully used to
produce 3-manifold invariants, see for example \cite{BBG17a, CGP14,
  GPT2}.  The chromatic morphisms are introduced
here.

\section{Main results and open problems}\label{Main results and open problems}

\subsection{The invariant $F'$}
We start with    notation used in the paper,
for    details see Section \ref{S:AlgSetup}.  Let $\cat$ be a pivotal $\FK$-category, where $\kk$ is a field.
Let~$F$ be the Penrose functor (defined using the Penrose graphical calculus) from the category of planar $\cat$-colored ribbon graphs to $\cat$, see for example \cite{GPV}.
Let~$\mt$ be a
modified trace (for shorteness, an m-trace) on an ideal $\I$ in~$\cat$.  The trace~$\mt$ induces an
invariant~$F'$ of $\I$-colored spherical graphs, see \cite{GPV}.  This invariant can be computed by composing $\mt$ and~$F$ on   cutting
presentations of the graphs (see Formula ~\eqref{E:DefF'}).

\renewcommand{\Po}{G}
We assume the m-trace $\mt$ to be  \emph{non-degenerate} in the sense that for any object
$P\in\I$, the pairing 
$$
\Hom_\cat(\unit, P) \times \Hom_\cat(P,\unit) \to \kk , \ \   (x,y)\mapsto \mt_P(xy)
$$
 is non-degenerate.  
Pick  a basis $\{x_i\}$ of $\Hom_\cat(\unit, P)$ and  let $\{y_i\}$ be the dual basis of
$\Hom_\cat(P,\unit)$ with respect to the above pairing.  
Set 
\begin{equation}\label{E:DefOmegaP}  \Omega_P= \sum_i x_i\otimes_\kk y_i  \in\Hom_\cat(\unit,P) \otimes_\kk \Hom_\cat(P, \unit )  
\end{equation}
where $\otimes_\kk$ is the usual   tensor product of vector spaces over $\kk$.
It is a standard fact from linear algebra that  $\Omega_P$ does not dependent on the choice of the basis
$\{x_i\}$ of $\Hom_\cat(\unit, P)$.

By a \emph{multi-handlebody} we   mean a disjoint union of a finite number of oriented  3-dimensional handlebodies.
  A \emph{$\cat$-colored ribbon graph} on a multi-handlebody~$H$ is a finite graph embedded  in $\partial H$  whose every  edge is colored with an object of~$\cat$ and every vertex is thickened in  $\partial H$  to a coupon colored with a
morphism of~$\cat$.   All coupons have a top and a bottom sides which in the  pictures will be the horizontal sides.  Since our graphs are drawn on a surface they have a natural framing and therefore can be considered as ribbon graphs in the usual sense.   When all the colors of  a $\cat$-colored ribbon graph are
in the ideal $\I$ we say that  the graph is  \emph{$\I$-colored}.  Let 
$ \Hp$ be the class of all pairs $(H,\Gamma)$ where $H$ is a multi-handlebody and $\Gamma$ is a  non-empty $\I$-colored graph on  $H$.   In the sequel  we extend the
colorings of coupons multilinearly.  In particular, this allows us to color an
ordered matching pair of coupons with the morphism
$\Omega_P=\sum_i x_i\otimes_\kk y_i $ as above.
We represent such a pair of  coupons with their adjacent edges  by the following figure:
\begin{equation}\label{E:OmegaDefMuilt}
\epsh{fig1}{20ex}=\sum_i\epsh{fig2}{20ex}\put(-92,-1){\ms{x_i}}\put(-37,5){\ms{y_i}}\;.
\end{equation}

We now  define a cutting operation  on colored graphs.
Let $(H,\Gamma)\in \Hp$ and let $D\subset H$ be an oriented properly embedded
disc  whose boundary
$\partial D \subset \partial H$ does not  meet  the coupons of $\Gamma$ and  intersects the edges of $\Gamma$
  transversely in a  non-empty set.
Cutting $(H,\Gamma)$ along~$D$ we obtain  a new multi-handlebody graph
$(\cut_D (H),\cut_D(\Gamma)) \in \Hp$ where $\cut_D(\Gamma)$ is obtained by  cutting the edges of
$\Gamma$ intersecting $\partial D$  and then joining the cut points into two new coupons in $\partial (\cut_D (H))$  (one  on each side of the
cut). The coupons are  colored  as in Formula \eqref{E:OmegaDefMuilt}, see the following figure:
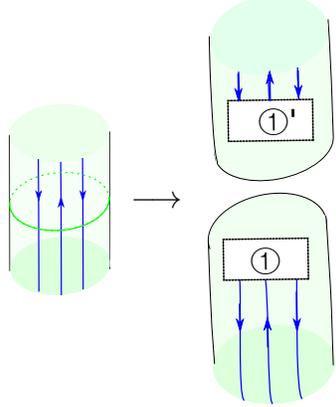
\begin{equation*}\label{E:CuttingPropF'}
\epsh{fig5-Francois}{14ex}\hspace{5pt} \longrightarrow \hspace{5pt} \epsh{fig6-Francois}{30ex}.
\end{equation*}
Note that $H$ and $\cut_D (H)$ can have different numbers of connected components and the orientation of $\cut_D (H)$ is induced by the one of $H$. The  cutting of $(H, \Gamma)$ at~$D$ is not determined by these data uniquely as it  depends on the order in the set $\Gamma \cap \partial D$ compatible with the cyclic order in the oriented circle~$\partial D$.

  Below we define an invariant  $F'$ of $\ideal$-colored  spherical graphs, see Formula~\eqref{E:DefF'}.
The following theorem (proved in Section \ref{S:InvOfProjHandlebodies}) extends~$F'$    to  $\Hp $.
\begin{theorem}\label{P:F'onProjHandlebodies}
Let $\cat$ be a pivotal $\FK$-category equipped with an ideal $\ideal$ in $\cat$  and a non-degenerate m-trace on   $\ideal$.
Then there exists a unique mapping $F':\Hp\to \FK$ 
satisfying the following four conditions.
\begin{description}
\item[(1) Invariance]
The element  $F'(H,\Gamma)$ of $\FK$  depends only on the orientation preserving diffeomorphism class of  $(H,\Gamma)\in \Hp$.
 
\item[(2) Spherical case] For any $\I$-colored ribbon graph $(B^3,\Gamma)$ 
on the 3-ball $B^3$, we have
 $
F'(B^3,\Gamma)=F'(\Gamma).$

\item[(3) Disjoint union] For any $(H_1,\Gamma_1), (H_2,\Gamma_2)\in \Hp$
we have $$F'(H_1\sqcup H_2,\Gamma_1 \sqcup \Gamma_2)=F'(H_1,\Gamma_1) F'(H_2,\Gamma_2).$$

\item[(4) Cutting] Cutting any $(H,\Gamma) \in \Hp$ along a 2-disc~$D$ as above, we always have 
  $F'(\cut_D(H),\cut_D(\Gamma))=F'(H,\Gamma)$.
\end{description}
\end{theorem}

The assumptions of Theorem~\ref{P:F'onProjHandlebodies} are rather mild.  By Theorem 5.5 of  \cite{GKP3}
(see also \cite{GKP2}) the categories arising  in   the following settings satisfy these assumptions:  representations of factorizable ribbon Hopf algebras, finite groups and their quantum doubles, Lie (super)algebras, the $(1,p)$ minimal model in conformal field theory, and
quantum groups at a root of unity.

Following the ideas of \cite{DGP}, we futher extend~$F'$ to so-called bichrome  handlebody graphs.  A \emph{bichrome handlebody graph} is a graph on the
boundary of a multi-handlebody
which is split as a disjoint union of two subgraphs, the 
blue and the red.  The blue subgraph is supposed to be $\cat$-colored while the red subgraph is supposed to be a  disjoint union of simple closed unoriented curves (which are not required to be $\cat$-colored).
We refer to these curves as \emph{red  circles}.   We call a bichrome handlebody graph   \emph{admissible}
if its blue subgraph is $\I$-colored and   meets each connected component
of the multi-handlebody at a non-empty set.

\begin{definition}\label{D:PurpleMap}  Let $G$ be an object of $\ideal$.  
  Set $\Lambda=\sum_i x_iy_i\in\End_\cat(G\otimes G^*)$ where
  $\Omega_{G\otimes G^*}=\sum x_i\otimes_\kk y_i$.  A \emph{chromatic morphism} for $G$
  is a morphism $\dd:G\otimes G\to G\otimes G$ such that
   \begin{equation}\label{E:HypothOfdd}
   (\Id_G\otimes\lev_G\otimes \Id_G)\circ(\Lambda\otimes\dd)
  \circ(\Id_G\otimes\rcoev_G\otimes\Id_\Po)=\Id_{G\otimes G}.
  \end{equation} \end{definition}
  This equality  is represented  pictorially as 
 $$
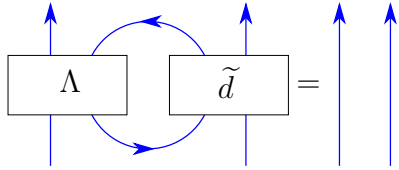
\epsh{fig19}{12ex}\put(-22,-2){$\wt d$}\put(-72,0){$\Lambda$}=\epsh{fig20}{12ex}
$$
where all blue strands are colored by $G$.
We utilize the word \emph{chromatic} here because the morphism $\dd$ is used to change a red
circle into a blue graph, see Theorem \ref{T:ExtenFtoRedBlue} below.  In the setting of Hopf algebras this transformation 
corresponds to  evaluation of the integral on the red circle (see Section \ref{SS:KuperbergInv}). 

A \emph{generator} of an  ideal $\I$ is an object $G\in\I$ such that for any
$P\in\I$,
\begin{equation}
  \label{eq:factorG}
  \Id_P=\sum_{j\in J} f_jg_j 
\end{equation}
for some   morphisms $f_j:G\to P,\,  g_j:P\to
G$ numerated  by a finite set $J$. 
As we will see below, in the setting of Hopf algebras such a generator  is  a projective generator.  

Let $\dd$ be a chromatic morphism for a generator $G$ of an ideal $\ideal$.   If $P\in\I$ 
and $\Id_P=\sum_jf_jg_j$ as in Formula \eqref{eq:factorG}, then we set 
\begin{equation}\label{eq:dp}
\dd_P=\sum_{j\in J}(\Id_G\otimes f_j)\dd(\Id_G\otimes g_j):G\otimes P\to G\otimes P.
\end{equation}
  The following theorem is proved  in
Section \ref{SS:ProofOfMainT}.
\begin{theorem}\label{T:ExtenFtoRedBlue}
  Let $\cat$ be a pivotal $\FK$-category equipped with a non-degenerate m-trace on an ideal $\ideal$ and a chromatic morphism $\dd$ on a generator $G\in\I$.  
    Then there exists a unique extension of   $F':\Hp\to \FK$  to admissible
  bichrome handlebody graphs
  which is preserved under the following transformation making  a red circle   blue in the presence of a  nearby  blue edge colored with an object $P\in\I$ as shown in the figure:
  \begin{equation}
   \label{E:ExtOfF'toRed}\epsh{fig3}{18ex}\put(0,15){\ms{P}}\longrightarrow\epsh{fig7}{18ex}\put(-18,-5){{$\dd_P$}}
    \put(-30,18){\ms G} \put(-13,25){\ms P} \;.
      \end{equation}
        Moreover,  if $(H,\Gamma)$ is a bichrome handlebody graph then $F'(H,\Gamma)$ only depends on the orientation preserving diffeomorphism class of $(H,\Gamma)$.
  \end{theorem}
Later we will see that $F'$ is also invariant under sliding an edge of $\Gamma$ over a red circle, see Proposition \ref{prop:sliding}.  

\subsection{The invariant $\Kuni$}\label{SS:InvKuni} 
Let $\cat$ be a pivotal $\FK$-category equipped with a non-degenerate m-trace
$\mt$ on an ideal~$\ideal$ and a chromatic morphism on a generator~$G$ of~$\ideal$.  
 Since $\mt$ is non-degenerate,  there exists a morphism $h:G \to G$ such that
$\mt_G(h)\neq 0$.  By renormalizing the m-trace we can assume $\mt_G(h)=1$.
Let $\IP$ be the ribbon graph in $\R^2$ formed by the braid closure of the
coupon filled with $h$.  Consider the bichrome handlebody graph $(B^3,\IP)$ where $\IP$ is colored in blue and viewed as a graph on   $S^2=\partial B^3$.   Clearly,  $F'(B^3,\IP)=F'(\IP)=\mt_G(h)=1$.

Let $M$ be a closed connected oriented 3-manifold.      A \emph{Heegaard diagram} of $M$ is a prescription for a Heegaard splitting $M=H_\alpha\cup_{\Sigma}H_\beta$ 
determined by upper and lower reducing sets of bounding circles $\{\alpha_i\}$ and $\{\beta_i\}$  in   $\Sigma=\partial H_\alpha=\partial H_\beta$, see Section \ref{S:InvBich3Man} for   details.
A Heegaard diagram as above  determines an admissible bichrome handlebody graph on $H_\alpha$ 
where the red subgraph is the set  $\cup_i \beta_i \subset \partial H_\alpha$ and the blue
subgraph is $\IP$ embedded in a small disc in $\partial H_\alpha$. We call this bichrome handlebody graph  a \emph{bichrome diagram} for $M$, see Figure \ref{F:L21} for an example.
\begin{figure}[htbp]
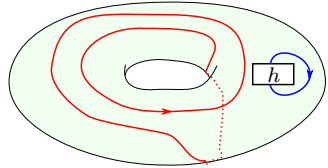

\begin{center}
$$
\epsh{fig33}{12ex}\put(-18,3){\ms{h}}
$$
\caption{\emph{Bichrome diagram for the lens space $L(2,1)$.}}
\label{F:L21}
\end{center}
\end{figure}

We state the main theorem of the paper   proved in Section \ref{SS:ProofOfMainT}.
\begin{theorem}
\label{T:MainExKup} 
   If $(H,\Gamma)$ is a bichrome diagram for a closed connected oriented 3-manifold $M$, 
   then
  $F'(H,\Gamma)\in \FK$  depends only of the diffeomorphism class of $M$.  
  \end{theorem}

We denote the invariant of Theorem \ref{T:MainExKup} by $\Kuni_\cat(M)$.  
Let us now discuss some examples.

\subsection{Hopf algebras and Kuperberg's Invariant}\label{SS:KupIntro} 
We use standard terminology from the theory of  Hopf algebras, for details see Section
\ref{SS:KuperbergInv}. Let $\Ho$ be a finite dimensional unibalanced
unimodular pivotal Hopf algebra.  Let $\Ho$-mod be the category of
finite dimensional modules over $\Ho$.  Let $\Proj$ be the ideal of
projective objects in $\Ho$-mod.  The Hopf algebra $\Ho$ itself with
its left regular representation is a generator of $\Proj$.
 
In Section \ref{SS:KuperbergInv} we prove the following theorem.
\begin{theorem}\label{T:ExHopf}
 There exist  a non-degenerate m-trace on $\Proj$ in $\Ho$-mod and a chromatic morphism $\dd$ on the generator $\Ho$ of $\Proj$.
 \end{theorem}
Let us briefly discuss the hypothesis on $\Ho$.  The existence of the chromatic morphism results from the theory of (co)integrals in  Hopf algebras.   The requirement that $\Ho$ is pivotal implies that the category $\Ho$-mod is pivotal.
The unimodularity of $A$ ensures that the ideal $\Proj$ has a non-degenerate right m-trace.
That $A$ is unibalanced implies that this m-trace is also a left m-trace, see \cite{BBG17b}.

Theorems \ref{T:MainExKup} and \ref{T:ExHopf} yield an invariant of 3-manifolds $\Kuni_{\Ho\text{-mod}}$.  By  Lemmas \ref{L:LP=LP'} and \ref{L:Def-d-forHopf}, the chromatic morphism~$\dd$ is essentially determined by the integral $\lambda:\Ho\to \FK$ and cutting along a bounding circle is determined by the cointegral $\Lambda\in \Ho$, where $\lambda(\Lambda)=1$.   

In Section \ref{SS:ProofKuperbergInv} we will prove the following theorem.

\begin{theorem}\label{T:K=KInvlutive}  If $\Ho$ is involutive
(i.e. the square of the antipode in~$A$ is the identity map) 
then $$\Kuni_{\Ho\text{-mod}}(M)=\Kup_\Ho(M)$$ where  $\Kup_\Ho(M)$ is the Kuperberg invariant derived from   $\Ho$, see \cite{Ku90}.  
\end{theorem}

\subsection{Turaev-Viro invariant}\label{SS:TVinvIntro} Let $\cat$ be
a finite semi-simple spherical $\FK$-category and let $ 
\{S_i\}_{i\in I}$ be a set of representatives of the isomorphism
classes of simple objects of $\cat$  (see
\cite{BW2}).  The  scalar
\begin{equation}
  \label{eq:dimC}
  \mathcal D =\sum_{i\in I} \qdim(S_i)^2 \in \FK 
\end{equation}   is called the dimension of $\cat$. We will explain  that $\cat$ satisfies the
assumptions  of Theorem \ref{T:MainExKup}, and, if  $\mathcal D \neq 0$, then the  resulting 3-manifold
invariant $\Kuni_\cat$ is equal to  the Turaev-Viro invariant $\TV$ \cite{TV,BW1}
associated to $\cat$.

Observe that here $G=\oplus_{i\in I} S_i$  is
a generator of $\ideal=\cat$ and 
the quantum trace $\mt=\operatorname{qTr}_\cat$ is
a non-degenerate m-trace on $ \cat$.  It follows that
$$\left\{x_i=\frac{1}{\qdim(S_i)}\lcoev_{S_i}\right\}_{i\in I} \text{ and } \left\{y_i=\rev_{S_i}\right\}_{i\in I}$$
 are dual  bases of $\Hom_\cat(\unit,G \otimes G^*)$ and $\Hom_\cat(G\otimes G^*,\unit)$, respectively. Using the expansion 
$\Omega_{G\otimes G^*}=\sum_{i\in I} x_i\otimes_\kk y_i$, it is straightforward to check that
$$\dd=\sum_{i\in I} \qdim(S_i)\Id_{S_i}\otimes \Id_G$$
is a chromatic morphism for
$G$.  
In Section \ref{S:ProofTVinv} we prove the following theorem.

\begin{theorem}\label{T:Kuni=TV}
If $\cat$ is a finite semisimple   spherical $\FK$-cate\-go\-ry of non-zero dimension~$\mathcal D $
with   chromatic morphism~$\dd$ and generator~$G$ then the invariant
$\Kuni_\cat$ is proportional to the Turaev-Viro invariant of  3-manifolds associated
to $\cat$:
$$\TV= {\mathcal D }^{-1} \, \Kuni_\cat.$$
\end{theorem}

\subsection{Open Problems}  Besides the categories  discussed here,  our approach certainly applies in other settings. Here we list  (from least   to most general)   three   further categories where our constructions should work:
\begin{enumerate}
\item the categories of finite dimensional modules over nice (quantum) Lie super algebras, see \cite{SchZ2001, SchZ2005},
\item the categories of finite dimensional modules over   nice quasi-Hopf algebras, see \cite{BC2003, BC2012,HN, PVO2000},
\item    general unimodular finite tensor categories, see \cite{Sh2017}.
\end{enumerate}
Here the adjective \lq\lq nice" means that the category satisfies the hypothesis of Theorem
\ref{T:ExtenFtoRedBlue}, in particular, admits an m-trace and a
chromatic morphism. The theory of \cite{GKP3} should imply the
existence of an m-trace in the above contexts though  we may need to choose
an appropriate pivotal structure to make the m-trace two sided, and it
may be useful to work  with the ideal of projective modules. It
seems likely that the references  listed  above  can help to construct 
chromatic morphisms in  the  categories  in question. It  is plausible that   the results of \cite{GKP2} may help to generalize our approach to
non-unimodular categories. 

In a different
direction, recall that 
Kuperberg \cite{Ku97} used framings of 3-manifolds  to generalize the invariant in \cite{Ku90} to arbitrary  finite dimensional Hopf algebras. As explained above,   a finite dimensional unibalanced unimodular pivotal Hopf algebra $\Ho$ gives rise to a framing-independent 3-manifold    invariant $\Kuni_{\Ho\text{-mod}}$ which   is computed in a way similar  to the invariants in \cite{Ku97} using an integral and a cointegral (see Section \ref{SS:KuperbergInv}). With Theorem \ref{T:K=KInvlutive} in mind, we   ask if $\Kuni_{\Ho\text{-mod}}(M)=\Kup_\Ho(M,f)$  for some framing~$f$ of a 3-manifold~$M$?  Is the Kuperberg invariant associated to a  {\em unibalanced unimodular pivotal Hopf algebra} framing-independent? (This is known  not be true for all  finite dimensional  Hopf algebras.)

\section{The algebraic setup}\label{S:AlgSetup}
\subsection{Pivotal and ribbon categories}\label{SS:LinearCat}
In this paper, we consider
strict tensor categories with tensor product $\otimes$ and unit object
$\unit$.  Let $\cat$ be such a category.  The notation $V\in \cat$
means that $V$ is an object of $\cat$.
 
The category $\cat$ is a \emph{pivotal category} if it has duality morphisms
\begin{align*} \coev_{V} &: 
\unit \rightarrow V\otimes V^{*} , & \ev_{V} & :  V^*\otimes V\rightarrow
\unit ,\\ 
 \tcoev_V & : \unit \rightarrow V^{*}\otimes V, &  \tev_V & :  V\otimes V^*\rightarrow \unit
\end{align*}
which satisfy compatibility
conditions (see for example \cite{BW2, GKP2}).

\subsection{$\FK$-categories}
Let $\FK$ be a field.  
A \emph{$\FK$-category} is a category $\cat$ such that its hom-sets are left $\FK$-modules, the composition of morphisms is $\FK$-bilinear, and the canonical $\FK$-algebra map $\FK \to \End_\cat(\unit), k \mapsto k \, \Id_\unit$ is an isomorphism.
A \emph{tensor $\FK$-category} is a tensor category $\cat$ such that $\cat$ is a $\FK$-category and the tensor product of morphisms is $\FK$-bilinear.

\subsection{M-traces on ideals in pivotal categories}\label{SS:trace}
Let $\cat$ be a pivotal $\FK$-category.  Here we recall the definition of an m-trace on an ideal in $\cat$, for more details see \cite{GKP2,GPV}.  
By an \emph{ideal} of $\cat$ we mean a full subcategory, $\ideal$, of~$\cat$ which is 
\begin{enumerate}
\item \textbf{Closed under $\otimes$:}  If $V\in \ideal$ and $W\in \cat$, then $V\otimes W$ and $W \otimes V $ are  objects of $\ideal$.
\item \textbf{Closed under retractions:} If $V\in \ideal$, $W\in \cat$  and there are  morphisms  $f:W\to V$,  $g:V\to W$ such that $g  f=\Id_W$, then $W\in \ideal$.
\end{enumerate}

An \emph{m-trace} on an ideal $\ideal$ is a family of linear functions
$$\{\qt_V:\End_\cat(V)\rightarrow \FK \}_{V\in \ideal}$$
such that following two conditions hold:
\begin{enumerate}
\item \textbf{Cyclicity:} 
 If $U,V\in \ideal$ then for any morphisms $f:V\rightarrow U $ and $g:U\rightarrow V$  in $\cat$ we have 
$
\qt_V(g f)=\qt_U(f  g)$.
\item \textbf{Partial trace properties:} \label{I:PartialTraceProp} If $U\in \ideal$ and $W\in \cat$ then for any $f\in \End_\cat(U\otimes W)$ and $g\in \End_\cat(W\otimes U)$ we have
\begin{equation*}
\qt_{U\otimes W}\bp{f}
=\qt_U\big((\Id_U \otimes \tev_W)(f\otimes \Id_{W^*})(\Id_U \otimes \coev_W)\big),
\end{equation*}
\begin{equation*}
\qt_{W\otimes U}\bp{g}
=\qt_U\big(( \ev_W\otimes \Id_U )( \Id_{W^*} \otimes g)( \tcoev_W\otimes \Id_U)\big).
\end{equation*}
\end{enumerate}

An $m$-trace on $\ideal$  is \emph{non-degenerate} if  for any  
$P\in\I$, the following pairing is non-degenerate:
$$
\Hom_\cat(\unit, P) \times \Hom_\cat(P,\unit) \to \kk, \,\,  (x,y)\mapsto \mt_P(xy).
$$
Using the pivotal structure and the partial trace property one can deduce from the non-degeneracy condition  that for all 
$P\in\I$ and $V\in \cat$, the pairing $$
\Hom_\cat(V, P) \times \Hom_\cat(P,V) \to \kk, \,\,  (x,y)\mapsto \mt_P(xy)
$$  is non-degenerate.

\subsection{Projective objects}\label{SS:ProjOb}
   An object $P$ of $\cat$ is \emph{projective} if  for any epimorphism $p\colon X \to Y$ and any morphism $f\colon P \to Y$ in $\cat$, there exists a morphism $g \colon P \to X$ in~$\cat$ such that $f=pg$.  An object $Q$ of $\cat$ is \emph{injective} if for any monomorphism $i\colon X \to Y$ and any morphism $f\colon X \to Q$ in $\cat$, there exists a morphism $g \colon Y \to Q$ in $\cat$ such that $f=gi$.  Denote by $\Proj$ the full subcategory of projective objects.   In a pivotal category projective and injective objects coincide (see \cite{GPV}).  Also, $\Proj$ is an ideal.

\subsection{Invariants of colored ribbon graphs}
Let $\cat$ be a pivotal $\FK$-category.  A morphism
$f:V_1\otimes{\cdots}\otimes V_n \rightarrow W_1\otimes{\cdots}\otimes
W_m$ in $\cat$ can be represented by a box and arrows:
$$ 
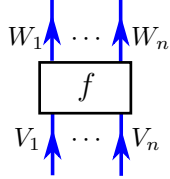
\epsh{fig32-francois}{13ex}\put(-37,-16){{\footnotesize $V_1$}}\put(0,-16){{\footnotesize $V_n$}}\put(-19,-16){{\footnotesize $\cdots$}}
\put(-39,16){{\footnotesize $W_1$}}\put(0,16){{\footnotesize $W_n$}}\put(-19,16){{\footnotesize $\cdots$}}\put(-17,0){$f$}
$$
which are called {\it coupons}. 
All coupons have a top and a bottom sides which in our  pictures will be the horizontal sides of the coupons.   By a ribbon graph in an
oriented manifold $\SS$, we mean an
oriented compact surface embedded in $\SS$ which is decomposed into
elementary pieces: bands, annuli, and coupons (see \cite{Tu}) and is a thickening of an oriented graph.  
In particular, the
vertices of the graph lying in $\Int \SS=\SS\setminus\partial \SS$ are
thickened to coupons.  A $\cat$-colored ribbon graph is a ribbon graph
whose (thickened) edges are colored by objects of $\cat$ and whose
coupons are colored by morphisms of~$\cat$.  
The intersection of a $\cat$-colored ribbon graph in $\Sigma$ with $\partial \Sigma$ is required to be empty or to consist only of vertices of valency~1.  
When $\SS$ is a surface the ribbon graph is just a tubular neighborhood of the graph.

A $\cat$-colored ribbon graph in $\R^2$ (resp.\ $S^{2}=\R^2\cup\{\infty\}$) 
is
called \emph{planar} (resp.\ \emph{spherical}).
Let $\Rib$ be the category of planar $\cat$-colored ribbon graphs and let $F:\Rib\to \cat$ be the pivotal functor\footnote{We call $F$ the pivotal functor because it can be  associated with each pivotal category.  Note, however, that the functor $F$ does not preserve the duality.} associated with $\cat$ via  the Penrose graphical calculus, see for example \cite{GPV}.    Let $\LL_{adm}$ be the class of all spherical $\cat$-colored ribbon graphs obtained as  the braid closure of a (1,1)-ribbon graph $T_V$ whose open edge is colored with  an object $V\in\ideal$.
 
 Given an m-trace $\mt$ on $\ideal$ we can renormalize $F$ to an invariant 
\begin{equation}\label{E:DefF'}
F':\LL_{adm}\to \kk \text{ given by } F'(L)=\mt_V(F(T_V))
\end{equation}
where $T_V$ is any (1,1)-ribbon graph as above.  The properties of the m-trace imply that $F'$ is a isotopy invariant of $L$, see \cite{GPV}.
\begin{remark}\label{R:fpp**}
If $\ideal$ is an ideal and $P\in \ideal$ then Lemma 2 of \cite{GPV} implies that $P^*\in \ideal$.  Moreover, the pivotal structure gives an isomorphism $f: P\to P^{**}$ for all $P$.  This isomorphism can be used to change the orientation of an edge of a graph as shown in the following diagram:
$$
\epsh{fig30-francois}{13ex}\put(-12,-1){{\footnotesize $P$}} \;\; \longleftrightarrow \;\;\epsh{fig31-Francois}{1in}\put(-15,12){{\footnotesize $f^{-1}$}}\put(-12,-12){{\footnotesize $f$}}\put(-6,-26){{\footnotesize  $P$}}\put(-6,-1){{\footnotesize $P^*$}}\put(-6,25){{\footnotesize $P$}}\;\;.
$$
\end{remark}

\section{Proof of Theorem \ref{P:F'onProjHandlebodies}}\label{S:InvOfProjHandlebodies}

\subsection{Preliminaries}  The following lemma contains standard facts from linear algebra; we leave the proof to the reader.  
\begin{lemma}\label{lem:basicalgebra}
Let $X_j$ and $Y_j$ be finite dimensional $\kk$-modules, for $j=1,2$.  Let
  $\brk{,}_{X_j,Y_j}:X_j\otimes_\kk Y_j \to \kk$ be a non-degenerate  bilinear pairing (i.e. a pairing  with trivial right and
  left kernels).  Given a basis $\{x^j_i\}$ of $X_j$  let 
   $\{y^j_i\}$ be the dual basis of $Y_j$ determined by $\brk{x^j_i,y^j_j}_{X_j,Y_j}=\delta_{i,j}$.   Then
  \begin{enumerate}
  \item the element $\Omega_j = \sum_i x_i^j \otimes_\kk y_i^j \in X_j\otimes Y_j$  is independent of the choice of the basis $\{x_i^j\}$,
  \item if $h:X_1\to X_2$ and $k:Y_2\to Y_1$ are $\kk$-linear maps such
  that $\brk{h(x),y}_{X_2,Y_2}=\brk{x,k(y)}_{X_1,Y_1}$ for all $x\in X_1$ and $y\in Y_2$ then 
  \begin{equation}\label{E:Omega12}
  (h\otimes \Id_{Y_1})( \Omega_1)=(\Id_{X_2}\otimes k)(\Omega_2).
  \end{equation}
  \end{enumerate}
\end{lemma}

\begin{proposition}\label{P:alg} Let $P\in \ideal$ and  $\Omega_P= \sum_i x_i\otimes_\kk y_i$ as in Formula  \eqref{E:DefOmegaP}. Let $\Lambda_P=\sum_ix_i y_i\in \End_\cat(P)$.  Then 
\begin{enumerate}  
\item \label{I:PalgIndep}  $\Omega_P$ is independent of the choice of the basis of $\Hom_\cat(\unit, P)$;  
\item \label{I:PalgNatural} If $P'\in \ideal$ and $\phi: P \to P'$ is a morphism then 
\begin{equation}\label{E:Omega12VT}
(\phi\otimes_\kk \Id_\unit)\Omega_P=\Omega_{P'}(\Id_\unit \otimes_\kk \phi)
\,\, \,\,  {\text {and}} \,\, \,\,
\phi \circ \Lambda_P=\Lambda_{P'}\circ \phi;
 \end{equation}
\item \label{I:Palg-fg} For any  morphisms  $f:\unit\to P $ and $g:P\to\unit$, we have
$$\mt_P(fg)=\sum_i \mt_P(fy_i) \, \mt_P(x_ig).$$
\end{enumerate}
\end{proposition}
\begin{proof}
  Set 
\begin{align*}
X_1&=\Hom_\cat(\unit, P), & Y_1&=\Hom_\cat( P,\unit),\\
X_2&=\Hom_\cat(\unit, P'),& Y_2&=\Hom_\cat( P',\unit).
\end{align*}  
For $j=1,2$, define the bilinear pairing $\langle ,\rangle_{X_j,Y_j}: X_j\otimes Y_j\to\kk$ by 
$$\langle x,y\rangle_{X_1,Y_1}=\mt_P(x\circ y), \text{ resp.\ } \langle x,y\rangle_{X_2,Y_2}=\mt_{P'}(x\circ y)$$
for $x\in X_j$ and $y\in Y_j$.  
Clearly,  Claim (1) of the proposition is a direct consequence of Claim (1) of Lemma \ref{lem:basicalgebra}. Next, we apply Claim (2)  of the lemma  to the linear maps 
$$h:\Hom_\cat(\unit,P)\to \Hom_\cat(\unit,P') \text{ given by } f\mapsto \phi\circ f,$$ 
$$k:\Hom_\cat(P',\unit)\to \Hom_\cat(P,\unit) \text{ given by } g\mapsto  g\circ \phi.$$ Then Formula  \eqref{E:Omega12} yields the first equality in~\eqref{E:Omega12VT}.   Writing this equality explicitly we get
$$
\sum_i \left(\phi\circ x_i\right)\otimes y_i=\sum_i x'_i\otimes \left(y'_i\circ \phi\right)
$$
where  $\Omega_{P'}= \sum_i x_i'\otimes_\kk y_i'$  and $\Lambda_{P'}=\sum_ix_i' y_i'\in \End_\cat(P')$ and the second equality in~\eqref{E:Omega12VT}  follows.  

To prove Claim  (3) of the proposition, expand $f=\sum a_i x_i$ and  $g=\sum b_iy_i$ where $a_i, b_i\in \kk$, $\{x_i\}$ is a basis of $X_1=\Hom_\cat(\unit, P)$ and   $\{y_i\}$ is the dual basis of $Y_1=\Hom_\cat(P, \unit)$ determined by the pairing 
$\langle ,\rangle_{X_1,Y_1}$.     
Then
$$ \mt_P(fg)=\sum_{i,j}a_ib_j\mt_P(x_i y_j)=\sum_{i}a_ib_i.$$
We also have   $\mt_P(f y_i)=\sum_j a_j \mt_P(x_j y_i)=a_i$ and similarly $\mt_P(x_ig)=b_i$ so Claim (3)  follows.
\end{proof}

\subsection{Proof of Theorem \ref{P:F'onProjHandlebodies}}

 By the genus of a multi-handlebody we mean the sum of the genera of   its components.  We proceed by induction  on the genus.
   The induction base concerns graphs on  disjoint unions of 3-balls.  Conditions (1) and (2) in Theorem~\ref{P:F'onProjHandlebodies}  define $F'$ uniquely for $\ideal$-colored graphs   on the boundary of a topological  3-ball.  Condition (3) extends 
  $F'$ uniquely to $\ideal$-colored graphs   on disjoint unions of 3-balls.  Note that cutting a 3-ball along a 2-disc as in Condition (4) produces two 3-balls.  To conclude the argument in the base case we
  check Condition  (4) for an arbitrary  
  $\ideal$-colored graph $(B^3,\G)\in \Hp$ on the boundary of a 3-ball $B^3$. 
  
   Let $D\subset B^3$ be an oriented   properly embedded disc   whose
  boundary  
  intersects $\G$ transversely (at the edges). 
Deforming if necessary~$D$, we can view  the circle  $\partial D$ as  
the   equator of the 2-sphere $ \partial B^3$ transversely  intersecting the edges of   $\G$ in several points.  We provide these points with  an order compatible with the cyclic order on the circle $\partial D$  whose orientation is induced by that of~$D$. Let $P\in \ideal$ be  the tensor product
  of the objects of~$\ideal$ associated with the edges of $\G$ traversing these points.     Let
  $\G_l$ and $\G_u$ be the parts of~$\G$ lying  in the lower and upper hemispheres,
  respectively (we use the orientation of $\partial D$ to distinguish them).  Cutting $(B^3,\G)$  along~$D$, we obtain an $\ideal$-colored graph on a
  disjoint union of two 3-balls $$(\cut_D(B^3),\cut_D(\G))=(B^3,\G'_l)\sqcup (B^3, \G'_u)$$ 
  where the new graphs are obtained by closing $\G_l$ and $\G_u$ with
  coupons using the tensor $\Omega_P=\sum_i x_i\otimes_\kk y_i$ as
  in Section~\ref{Main results and open problems}.
Consider the morphisms
  $F(\G_l):\unit \to P$ and $F(\G_u):P\to \unit$. By the
  definition of~$F'$ and Proposition \ref{P:alg}, Claim \eqref{I:Palg-fg}, we have 
  $$ F'(B^3,\G)=\mt_P(F(\G_l)F(\G_u)) =\sum_i \mt_P(F(\G_l)y_i) \, \mt_P(x_iF(\G_u)).
  $$
  Since, by definition, $F'(\G'_l)=\mt_P(F(\G_l)y_i)$ and
  $F'(\G'_u)=\mt_P(x_iF(\G_u))$, this computation verifies  
 Condition  (4) for   $(B^3,\G)$ and~$D$. 
  This concludes the proof of the  induction
 base.

  Assume now that we have an integer $g\geq 1$   and an invariant $F'$ defined for all colored ribbon graphs on  multi-handlebodies  of
  genus $<g$ and  satisfying Conditions (1)-(4). 
   Let $(H,\G)\in \Hp$ where $H$ is a multi-handlebody of genus~$g$.   Pick an oriented properly embedded disc $D\subset H$     such that  $\partial D$ is an essential simple closed curve   on $\partial H$ missing all the coupons and  intersecting all the edges of~$\G$ transversely.  Cutting $(H,\G)$ along~$D$ produces a 
  multi-handlebody $(\cut_D(H),\cut_D(\G))$ of genus $g-1$ which we  denote simply $\cut_D(\G)$.  Set $F'(H,\G)=F'(\cut_D(\G))$.  We claim that  $F'(H,\G)$ is well defined, i.e.\ is independent of the choice of the disc~$D$.  Pick another oriented properly embedded disc   $D'\subset H$    as above.
  We will prove   that $F'(\cut_D(\G))=F'(\cut_{D'}(\G))$.   
Assume first $\partial D\cap \partial D'=\emptyset$.  Deforming, if necessary, $D'$ in~$H$, we can   assume that  $D\cap D'=\emptyset$. Then 
  $$F'(\cut_{D'}(\G))=F'(\cut_D(\cut_{D'}(\G)))=
  F'(\cut_{D'}(\cut_{D}(\G))) =F'(\cut_D(\G))$$ where
  the first and the last equalities hold because, by the induction assumption, $F'$ satisfies  Condition (3)  
  in   genera $<g$.
  (This in particular proves that $F'(\cut_D(\G))$ does not depend on the orientation of $D$: choose $D'$ to be parallel to $D$ but oriented in  the opposite way.) If $\partial D\cap \partial D' \neq \emptyset$, then  it follows from \cite{MM}, that there is a sequence
  $D_0=D, D_1,\ldots, D_n=D'$ of oriented properly embedded discs in~$H$ bounded by essential closed curves in $\partial H$ and such that
$D_i\cap D_{i+1}=\emptyset$ for $i=0,1, ..., n-1$.
By the above,
  $$  F'(\cut_{D_0}(\G))=F'(\cut_{D_1}(\G))=\cdots=F'(\cut_{D_n}(\G))  .$$ It is clear now that  $F'$ is defined for all colored ribbon graphs on  multi-handlebodies  of
  genus $<g+1$ and  satisfies Conditions (1)-(4).  In particular, to check Condition (1), consider an  orientation preserving diffeomorphism $f: (H, \Gamma)\to (H', \Gamma')$ of colored ribbon graphs on multi-handlebodies of genus $\leq g$. Pick a disc $D\subset H$ as above.  By definition and induction  assumption,
 $$F'(H,\G)=F'(\cut_D(\G)) =F'(\cut_{f(D)}(\G')) = F'(H',\G').$$ This completes the induction step.

 \section{Invariants of bichrome graphs and   3-manifolds}\label{S:InvBich3Man}
\subsection{Bichrome  handlebody graphs}
For   reader's convenience, we   summarize here  the terminology introduced above. A \emph{multi-handlebody} is a disjoint union of a finite number of oriented 3-dimensional handlebodies.  
 The \emph{genus} of a multi-handlebody is  the sum of the genera of   its components.   A \emph{bichrome handlebody graph} is a pair $(H,\Gamma)$ where $H$ is a multi-handlebody and   $\Gamma$ is a  finite graph on $\partial H$ considered up to isotopy and formed by two disjoint subgraphs: $\Gamma=\Gamma_{blue}\sqcup \Gamma_{red}$ where $\Gamma_{blue}\subset \partial H$ is a  $\cat$-colored graph and $\Gamma_{red}\subset \partial H$ a union of unoriented disjoint  simple closed  curves called \emph{red circles}. A bichrome handlebody graph $(H,\Gamma)$ is \emph{admissible}
if its blue subgraph is $\I$-colored and meets each connected component
of~$H$.     We will sometimes  drop~$H$ and denote a bichrome handlebody graph $(H,\Gamma)$ simply by  $\Gamma$.
Finally, we let   $\Hb$ denote the set of orientation preserving
diffeomorphism classes of admissible bichrome handlebody graphs. 
\begin{definition} Let $H$ be a multi-handlebody.
  \begin{enumerate}
   \item By a \emph{circle} in $\partial H$ we  mean an (unoriented) simple closed curve in $\partial H$.
   
   \item We say that a set of disjoint   circles   in $\partial H$ \emph{bounds} in $H$ if these circles bound a disjoint union of discs embedded in $H$.
      \item A bounding set of circles  in $\partial H$ is a \emph{reducing set} if the
     complement of these circles  in $\partial H$  is a disjoint union of 2-spheres with holes.
  \item The \emph{complexity} of a circle $\gamma$ in $\partial H $ is the minimal
    number of intersections of $\gamma$ with a  reducing set of
    circles.
\end{enumerate}
\end{definition}

\begin{definition}[Red Capping and Digging moves]\label{def:digcap}
 Let  $(H,\Gamma)$ and $ (H',\Gamma')$ be bichrome handlebody graphs.  
We say that $(H,\Gamma)$ is obtained from $(H',\Gamma')$ by a \emph{red capping move}  along a red circle $c\subset \Gamma'_{red}$ if 
\begin{enumerate}
\item there is  a properly embedded  disc $D\subset  H'$ such that the circle $\partial D \subset \partial H'$  transversely intersects~$c$ in one point, 
\item $H$ is obtained from $H'$ by attaching a $2$-handle along $c$, and 
\item $\Gamma=\Gamma'\setminus c$, where we identify the complement in $\partial H' $ of an annulus neighborhood of~$c$ with a subset of $\partial H$.
\end{enumerate}
Conversely, we say $(H',\Gamma')$ is obtained from $(H,\Gamma)$ by a \emph{red digging move} (producing the red circle $c$), see Figure \ref{F:dig}.   
\end{definition}
\begin{figure}
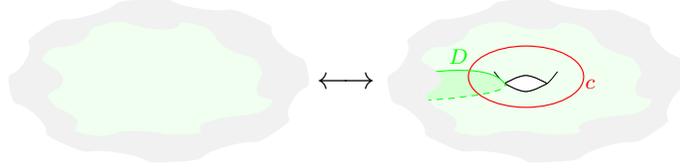

\begin{center}
$$
\epsh{fig17}{12ex}
\longleftrightarrow \epsh{fig18b}{12ex}\put(-32,0){\color{red}{\ms{c}}}\put(-75,8){\color{green}{\ms{D}}}
$$
\caption{Red Digging / Capping moves}
\label{F:dig}
\end{center}
\end{figure}

\begin{proposition}\label{prop:complexityonedigging} 
A red circle in a bichrome handlebody graph is the result of a red digging move if and only if it has complexity one.
\end{proposition}
\begin{proof}
Let $c$ be a red circle in a bichrome handlebody graph $(H',\Gamma')$.  If $c$ has complexity one then there is a reducing set $S$ of circles  in $\partial H'$  such that $c\cap S=\{pt\}$.
Attaching a 2-handle to $H'$  along $c$, we get  a multi-handlebody~$H$, and the graph $ (H,\Gamma'\setminus c)$ is obtained from $(H',\Gamma')$ by a red capping move. Its  inverse is the required  digging move. 
Conversely, if $c$ is obtained by a red digging move then it  intersects exactly once the circle $\partial D$ in Definition \ref{def:digcap}.
\end{proof}
\begin{proposition}\label{prop:sliding}
  Let $(H,\Gamma), (H,\Gamma') \in\Hb$ be bichrome graphs such that $\Gamma'$ is obtained from $\Gamma$  by sliding a blue or red edge of $\Gamma$ over one of its  red circles.  Then $\Gamma$
  and $\Gamma'$ are related by a sequence of red digging and capping moves.
\end{proposition}
\begin{proof}
Let $\Gamma=\Gamma_{blue}\cup\Gamma_{red}$ and let $c\subset \Gamma_{red}$ be  a red  circle on which we want to slide an edge $e$ of   $\Gamma $. 
Claim:  up to applying one red digging move which transforms $(H,\Gamma)$ into a new bichrome graph $(H_1,\Gamma_1)$ and creates a new red circle $c'\subset \Gamma_1$, we can reduce ourselves to the case where $c$ is a red circle created by a red digging move. 
This claim would imply the proposition because after applying a red capping move on $c$, we get a bichrome graph $(H_2,\Gamma_2)$ in which the edge $e$ can be slid, by an isotopy, over the disc added by the red capping. Re-digging along the same disc and re-capping along $c'$, we get   the bichrome graph $(H,\Gamma')$ obtained by sliding $e$ over $c$.  

The claim above is proved  by describing a suitable red digging move. 
Let $I\subset \partial H$ be a parametrized segment, i.e.\ the image of an embedding $i:[-1,1]\hookrightarrow\partial H$ such that $\partial I\cap \Gamma=i(\{0\})$ is formed by a single point belonging to $c$.  Let $I'$ be a properly embedded arc in $H$ obtained by slightly pushing $I$ inside $H$ (keeping $\partial I$  fixed).  Also, set 
$$
H_1=H\setminus Tub(I'), \;\; c'=\partial Tub(i(1)) 
, \;\;
(\Gamma_1)_{blue}=\Gamma_{blue}, \;\;  (\Gamma_1)_{red}=\Gamma_{red} \sqcup c'
$$
where  $Tub$ stands for  a tubular neighborhood and we view $\partial H\setminus Tub(\{i(\pm1)\})$ as  a subset of $\partial H_1$.
Then $(H_1,\Gamma_1)$ is obtained from $(H,\Gamma)$ by a red digging move along $I$ which creates the red circle $c'$.   
But, in the notation of Definition \ref{def:digcap}, $c$ intersects the disc $D$ bounded by $I\cup I'$ exactly once so it has complexity~$1$. By Proposition \ref{prop:complexityonedigging}, $c$ is created by a red digging  move.
\end{proof}

\subsection{Heegaard splittings}\label{SS:HeegSplitting}
A \emph{Heegaard splitting} for a closed  connected oriented 3-manifold $M$ is an ordered triple $(H_\alpha,H_\beta,\Sigma)$ where: (1) $H_\alpha$ and $H_\beta$ are   handlebodies   embedded in $M$ and endowed with orientation induced from that of~$M$, (2)  $H_\alpha \cap H_\beta= \partial H_\alpha = \partial H_\beta$; (3) the surface $\Sigma=\partial H_\alpha = \partial H_\beta $   is oriented as the boundary of $H_\alpha$ (with the outgoing vector first convention).  Note that under these assumptions we have  $ H_\alpha \cup H_\beta=M$.
A \emph{Heegaard diagram} of $M$ compatible with the Heegaard splitting $(H_\alpha,H_\beta,\Sigma)$ is a
triple $(\Sigma, \{\alpha_i\}, \{\beta_i\})$ where $\{\alpha_i\}$ and
$\{\beta_i\}$ are minimal reducing sets of circles in~$\Sigma$ bounding  in $H_\alpha$
and $H_\beta$, respectively.
Clearly,  from the Heegaard diagram  one can recover $H_\alpha$, $H_\beta$,
and $M= H_\alpha\cup_\Sigma H_\beta$ up to diffeomorphism.

By Section \ref{SS:InvKuni},   a Heegaard diagram
$(\Sigma,\{\alpha_i\}, \{\beta_i\})$ of~$M$ determines a bichrome diagram for~$M$ represented by the bichrome
handlebody graph $(H_\alpha,\Gamma_{blue}\cup \Gamma_{red})$ where
$\Gamma_{red}=\{\beta_i\}$ and $\Gamma_{blue}$ is the planar ribbon
graph $ \IP$ which is the braid closure of the coupon filled with $h$
such that $\mt_G(h)=1$.  (Since $H_\alpha\setminus \Gamma_{red}$ is connected,  the position of  $ \IP$ is unique up to isotopy.)

\begin{theorem}\label{T:HandlebodyDiagsM}
Any two bichrome diagrams for $M$ are related by a finite sequence of red digging and capping moves.
\end{theorem}
\begin{proof}
Let  $(H,\Gamma)$ and $(H',\Gamma')$ be bichrome diagrams for $M$. It is well known that $(H,\Gamma_{red})$ and $(H',\Gamma'_{red})$
  can be related by a finite sequence of stabilization moves, 
  the  inverse moves, and  handle slide moves (see \cite{Si}).
  A stabilization move is a special case of a red digging move
 (and its inverse is a red capping move). By Proposition
  \ref{prop:sliding}, handle slides are compositions of  isotopies with  red
  digging and capping moves. All these moves can be applied here because $\Gamma_{blue}=\IP$ lies  in a disc in $\partial H$.
\end{proof}

\subsection{Proof of Theorems \ref{T:ExtenFtoRedBlue} and \ref{T:MainExKup}}\label{SS:ProofOfMainT}
We start with  lemmas.  
\begin{lemma}\label{L:HypothOfdd_P} Let $P,P'\in \ideal$ and $$\Lambda_{P'\otimes G^*}=\sum_i y_ix_i\in\End_\cat(P'\otimes G^*)\,\, \, {\rm {where}} \,\,\, \Omega_{P'\otimes G^*}=\sum x_i\otimes_\kk y_i.$$ Then 
 \begin{equation}\label{E:HypothOfdd_P}(\Id_{P'}\otimes\lev_G\otimes\Id_P)\circ(\Lambda_{P'\otimes G^*}\otimes\dd_P)
  \circ(\Id_{P'}\otimes\rcoev_G\otimes\Id_P)=\Id_{P'\otimes P}.
  \end{equation}
\end{lemma}
\begin{proof}
Since $G$ is a generator of $\I$ there exist $g_i:P\to G, f_i:G\to P$ and $g'_i:P'\to G, f'_i:G\to P'$ such that $\sum_i f_i\circ g_i=\Id_P$ and $\sum_i f'_i\circ g'_i=\Id_{P'}$. 
Then  
\begin{align*}
\Id_{P'\otimes P}&=\sum_{i,j} (f'_i\otimes f_j)\circ \Id_{G\otimes G}\circ (g'_i\otimes g_j)\\
&=\sum_{i,j}  (f'_i\otimes\lev_G\otimes f_j)\circ(\Lambda\otimes\dd)
  \circ(g'_i\otimes\rcoev_G\otimes g_j)\\
&=\sum_{i} (f'_i\otimes\lev_G\otimes \Id_P)\circ(\Lambda\otimes\dd_{P})
  \circ(g'_i\otimes\rcoev_G\otimes\Id_P)\\
&  =\sum_{i} (\Id_{P'}\otimes\lev_G\otimes \Id_P)\circ 
 \left(\left((f'_i\otimes \Id_{G^*})\Lambda (g'_i\otimes\Id_{G^*})\right)\otimes\dd_{P}\right) \circ
 \\& \hspace{50pt}\circ (\Id_{P'}\otimes\rcoev_G\otimes\Id_P)\\
&  =\sum_{i} (\Id_{P'}\otimes\lev_G\otimes \Id_P)\circ \left(\left(\Lambda_{P'\otimes G^*}\circ(f'_ig'_i\otimes \Id_{G^*})\right)\otimes\dd_{P}\right)\circ
\\
 & \hspace{50pt}\circ (\Id_{P'}\otimes\rcoev_G\otimes\Id_P)\\
&  =\sum_{i} (\Id_{P'}\otimes\lev_G\otimes \Id_P)\circ(\Lambda_{P'\otimes G^*}\otimes\dd_{P})
  \circ(\Id_{P'}\otimes\rcoev_G\otimes\Id_P)
  \end{align*}
  where the second equality uses the definition of the chromatic morphism (Formula   \eqref{E:HypothOfdd}), the third equality uses the definition of $\dd_{P}$ (Formula  \eqref{eq:dp}) and 
  the second to last equality comes from Part \eqref{I:PalgNatural} of Proposition \ref{P:alg} where $\phi=f'_i\otimes \Id_{G^*}$.
    \end{proof}

Given a bichrome handlebody graph $(H,\Gamma)$  we can produce a new bichrome handlebody graph $(H',\Gamma')$ by doing a red digging move on $(H,\Gamma)$ then changing the newly created red circle into a blue graph using the chromatic morphism, as in   Formula  \eqref{E:ExtOfF'toRed}.
We say $(H',\Gamma')$ is obtained from $(H,\Gamma)$ by a \emph{blue digging move}.  Conversely, we say $(H, \Gamma)$ is obtained from $(H',\Gamma')$ by a \emph{blue capping move}.    
See Figure \ref{F:BlueDigging} for a pictorial representation of these moves.   
\begin{figure}[htbp]
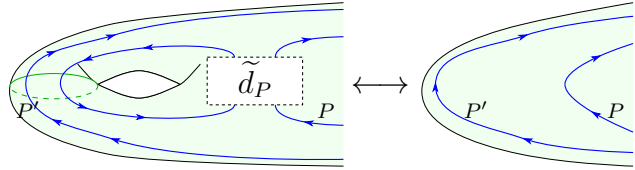

\begin{center}
$$
\epsh{fig15-Francois}{12ex}\put(-8,-8){\ms{P}}\put(-33,0){$\dd_P$}\put(-103,-8){\ms{P'}}
\longleftrightarrow \epsh{fig16-francois}{12ex}\put(-8,-8){\ms{P}}\put(-53,-8){\ms{P'}}
$$
\caption{\emph{Blue Digging and Capping Moves.} Here one can assume the orientation of the leftmost strand is as shown in the figure, cf. Remark \ref{R:fpp**}.}
\label{F:BlueDigging}
\end{center}
\end{figure}

\begin{lemma}\label{P:F'InvIcolBlueDigging}
The invariant $F'$ of $\I$-colored ribbon graphs on multi-handlebodies defined in Theorem  \ref{P:F'onProjHandlebodies} is preserved  under blue digging and capping moves.  
\end{lemma}
\begin{proof} Let $(H, \Gamma)$ and  $(H',\Gamma')$ be bichrome handlebody graphs such that $(H', \Gamma')$ is obtained from $(H,\Gamma)$ by a blue digging move.  
Up to 
cutting along a  reducing set of curves for $H$, we can assume that the component of $H$ to which we are applying the blue digging has genus $0$. 
To prove the lemma we  compute the morphisms associated to the subsurfaces of $\partial H$ and $\partial H'$ drawn in Figure \ref{F:BlueDigging}. 
Clearly, the   morphism  associated  with the right hand side is   $\Id_{P'}\otimes \Id_{P}$.  
To compute the morphism, say $f$, appearing  on the left hand side, let $\gamma$ be the curve bounding the disc in the far left part of Figure \ref{F:BlueDigging}. 
Cutting along $\gamma$, as discussed before Theorem \ref{P:F'onProjHandlebodies}, we get that  
$$f=(\Id_{P'}\otimes\lev_G\otimes\Id_P)\circ(\Lambda_{P'\otimes G^*}\otimes\dd_P)
  \circ(\Id_{P'}\otimes\rcoev_G\otimes\Id_P).$$
By Lemma \ref{L:HypothOfdd_P}, $f=\Id_{P'}\otimes \Id_{P}$.
\end{proof}

\begin{proof}[Proof of Theorem \ref{T:ExtenFtoRedBlue}]
We  prove the following slightly stronger claim.

\begin{description}
\item[Claim] There exists a unique extension of $F'$ to an invariant of admissible
 bichrome handlebody graphs
 satisfying Formula  \eqref{E:ExtOfF'toRed} and invariant under the blue capping and digging moves.  
\end{description}
We prove this  by induction on the number of red circles.  When there are no red circles the claim is just Lemma \ref{P:F'InvIcolBlueDigging}.  
Now suppose that the claim holds  for all admissible bichrome handlebody graphs with $n-1$ red circles.  Let $(H,\Gamma)$ be an admissible bichrome handlebody graph with $n$ red circles.  Pick an injective  path $\gamma_0$ in $\partial H \setminus \Gamma$ leading  from a point on a blue edge to a point on a red circle $\beta_0$.  Using  $\gamma_0$,  we  pull a small segment of the blue edge to $\beta_0$ then use the chromatic morphism to change $\beta_0$ into a blue graph.  We obtain an admissible bichrome handlebody graph $(H,\Gamma_{\gamma_0})$ with $n-1$ red circles.  By the induction assumption, $F'(H,\Gamma_{\gamma_0})$ is well defined. So,  if $F'$ exists, then  it is unique. To prove that $F'$ exists we  show its independence  of the choice of $\gamma_0$.    Let $\gamma_1$ be another such path going from a point on a blue edge to a point on a red circle $\beta_1$ and let $(H,\Gamma_{\gamma_1})$ be the admissible bichrome handlebody graph obtained as above using $\gamma_1$ to make $\beta_1$ blue.   
We separate two cases.\\

\noindent \textbf{Case 1.} $\beta_0=\beta_1$.  Here we have two sub-cases.  
First, suppose the red circle $\beta_0=\beta_1$ has complexity one.  By Proposition \ref{prop:complexityonedigging}, the circle  $\beta_0$ is the result of a digging move.  Therefore, when we use   $\gamma_0$ or $\gamma_1$ and the chromatic morphism to change $\beta_0$ into a blue graph, we arrive at a diagram as in the left side of Figure \ref{F:BlueDigging}.  In both cases we can do a blue capping move to arrive at the same $\I$-colored ribbon graph on a multi-handlebody with  $n-1$ red circles.   Thus, induction implies
that $
F'(H,\Gamma_{\gamma_0})=F'(H,\Gamma_{\gamma_1})
$. 
So, in this case the extension of $F'$ does not depend on the choice of $\gamma_0$.  

Second, assume the red curve $\beta_0=\beta_1$ has any complexity.
Apply a blue digging move to $\Gamma$ along a small interval $I$ intersecting $\beta_0$ in a single point.   The result is a new bichrome handlebody graph $(H',\Gamma')$ with $n$ red circles in which the image $\beta_0'$ of $\beta_0$ is a red curve with complexity $1$. Since this digging move only modified $(H,\Gamma)$ in a neighborhood of $I$, we can identify $\gamma_0$ and $\gamma_1$ as paths in $\partial H'\setminus \Gamma'$ and we can apply the chromatic morphism to $\beta'_0$ either through $\gamma_0$ and $\gamma_1$ getting,  respectively, bichrome handlebody graphs $\Gamma'_{\gamma_0}$ and $\Gamma'_{\gamma_1}$  with $n-1$ red components. By the preceding case $F'(H',\Gamma'_{\gamma_0})=F'(H',\Gamma'_{\gamma_1})$.
Observing that $\Gamma_{\gamma_0}'$ (resp. $\Gamma_{\gamma_1}'$) is obtained from $\Gamma_{\gamma_0}$ (resp. $\Gamma_{\gamma_1}$) by a blue digging move,  we get
$$F'(H,\Gamma_{\gamma_0})=
F'(H',\Gamma'_{\gamma_0})=F'(H',\Gamma'_{\gamma_1})=F'(H,\Gamma'_{\gamma_1}).$$

\noindent \textbf{Case 2.}  $\beta_0\neq \beta_1$.  We separate  two sub-cases.  First, suppose $\gamma_0$ and $\gamma_1$ are disjoint.
In $(H,\Gamma_{\gamma_0})$ (resp.\ $ (H,\Gamma_{\gamma_1})$) we can use  $\gamma_1$ (resp.\ $\gamma_0$) to change $\beta_1$ (resp.\ $\beta_0$) into a blue graph and obtain an admissible bichrome handlebody graph $(H,\Gamma_1)$ with $n-2$ red circles.  By the induction assumption, 
$$
F'(H,\Gamma_{\gamma_0})=F'(H,\Gamma_1) = F'(H,\Gamma_{\gamma_1}).
$$ 

Second, suppose $\gamma_0\cap \gamma_1\neq \emptyset$. We claim that there is another path $\gamma'_1$ connecting $\Gamma_{blue}$ to $\beta_1$ such that $\gamma_0\cap \gamma'_1=\emptyset$. Then   by the previous  subcase  $F'(H,\Gamma_{\gamma_0})=F'(H,\Gamma_{\gamma'_1})$ and  by Case 1  we have  $F'(H,\Gamma_{\gamma'_1})=F'(H,\Gamma_{\gamma_1})$. So,  the result  follows.  
To prove our claim observe that since  $\gamma_0\cap \gamma_1\neq \emptyset$. the paths $\gamma_0$ and $\gamma_1$ lie  in the same connected component $R$ of $\partial H\setminus \Gamma$.  Moreover, $R$ is an open orientable surface which is the interior of a compact surface with at least $3$ distinct boundary components: $\partial_{blue}\subset \Gamma_{blue}$,  $\beta_0$ and $\beta_1$.  
Then $\gamma_0$ and $\gamma_1$ are embedded arcs in $R$ connecting $\partial_{blue}$ to $\beta_0$ and $\beta_1$, respectively.  But $R\setminus \gamma_0$ is connected as $\gamma_0$ intersects the closed curve $\beta_0$ once.  Thus, there exists another path in $R\setminus \gamma_0$ connecting $\partial_{blue}$ and $\beta_1$.
\end{proof}

\begin{proof}[Proof of Theorem \ref{T:MainExKup}]
By Theorem \ref{T:HandlebodyDiagsM} it is enough  to show that $F'$ is invariant under red digging and capping moves. 
Let $(H,\Gamma)$ and $(H',\Gamma')$ be bichrome handlebody graphs such that $(H',\Gamma')$ is obtained from $(H,\Gamma)$ by a red digging move.  Suppose that $c$ is the red circle created by  this move.  Let $(H',\Gamma'')$ be the bichrome handlebody graph obtained from using an edge of the blue graph and the chromatic morphism to change the red circle $c$ into a blue graph.  By definition, the composition of these two moves is a blue digging move.  By Lemma \ref{P:F'InvIcolBlueDigging} we have 
 $F'(H,\Gamma)=F'(H',\Gamma'')$. 
Since $(H',\Gamma')$ and $(H',\Gamma'')$ differ by an isotopy and a move represented in Formula  \eqref{E:ExtOfF'toRed},  Theorem \ref{T:ExtenFtoRedBlue} implies  that $F'(H',\Gamma')=F'(H',\Gamma'')$. 
Combining these equalities we get $F'(H,\Gamma)=F'(H',\Gamma')$ which concludes the proof.
\end{proof}

\section{Hopf Algebras and the Kuperberg invariants} \label{SS:KuperbergInv}

In this section we prove Theorems \ref{T:ExHopf} and \ref{T:K=KInvlutive}.  First, we briefly recall some well known facts about Hopf algebras, see for
example \cite{Rad2011}.  

\subsection{Hopf algebra preliminaries} Let $\Ho$ be a finite-dimensional Hopf algebra over a field $\FK$ 
with  multiplication $m : \Ho \otimes \Ho \rightarrow \Ho$,  unit
$\eta : \Bbbk \rightarrow \Ho$,  coproduct
$\Delta : \Ho \rightarrow \Ho \otimes \Ho$,  counit
$\epsilon : \Ho \rightarrow \Bbbk$,  antipode $S : \Ho \rightarrow \Ho$.
A \emph{right integral} of $\Ho$ is a linear form $\lambda \in \Ho^*$ such that 
$\lambda f = f(1_\Ho) \cdot \lambda$ for all  $f \in \Ho^*$.  This means
that
$(\lambda \otimes \Id_\Ho)(\Delta(x))=\lambda(x) \cdot 1_\Ho$ for every
$x \in \Ho$.  A \emph{left (resp.\ right) cointegral} of $\Ho$ is a vector $\Lambda \in \Ho$
satisfying $x \Lambda = \epsilon(x) \Lambda$ (resp.\ $ \Lambda x= \epsilon(x) \Lambda$) for all  $x \in \Ho$.
Since $\Ho$ is finite-dimensional, right integrals form a 1-dimensional
ideal in $\Ho^*$ and left cointegrals form a 1-dimensional ideal in
$\Ho$. Moreover, every non-zero right integral $\lambda \in \Ho^*$ and
every non-zero left cointegral $\Lambda \in \Ho$ satisfy
$\lambda(\Lambda) \neq 0$. We fix a   right integral $\lambda \in \Ho^*$ and   a left
cointegral $\Lambda \in \Ho$ satisfying $\lambda(\Lambda) = 1$.   We will  use sumless Sweedler's notation to describe the   coproduct, for example, we write  $\Delta^3(x)=x_{(1)}\otimes x_{(2)}\otimes x_{(3)}$.

The Hopf algebra $\Ho$ is \textit{unimodular} if
$S(\Lambda) = \Lambda$, or equivalently, if $\Lambda$ is both a right and left cointegral.  We call  $\Ho$  \emph{pivotal} if there exists $g\in \Ho$ such that $S^2(x) = gxg^{-1}$  for all $x \in \Ho$.   Let $\Proj$ be the ideal of projective modules over $\Ho$ (for the definition of projective objects,  see Section \ref{SS:ProjOb}).  By Theorem~1 of \cite{BBG17b} every finite dimensional unimodular pivotal Hopf algebra has a left m-trace on $\Proj$.  Such a Hopf algebra is \emph{unibalanced} if this left m-trace is also a right m-trace.
 
\subsection{Proof of Theorem \ref{T:ExHopf}} \label{SS:ProofOfThmEXHopf} Let $\Ho$ be a finite dimensional unibalanced unimodular pivotal Hopf algebra over a field $\Bbbk$ and let $\Ho$-mod be the category of its finite dimensional left modules.  We denote by the same symbol~$\Ho$   the algebra~$\Ho$  viewed as a left $\Ho$-module   via  the action $L : \Ho \to \End_\kk(\Ho)$ given by 
$L_h(x) = hx$ for all $h,x \in \Ho$.

\begin{lemma}
The  left $\Ho$-module $\Ho$ is a generator of $\Proj$ in $\Ho$-mod.
\end{lemma}
\begin{proof}
Let $P$ be an indecomposable projective $\Ho$-module.  Then $P$  is a direct summand of a  direct sum of a finite number of copies of~$A$.  Since $P$ is indecomposable,  $P$ is a direct summand of $\Ho$.  Since  the  Krull-Schmidt theorem holds in $\Ho$-mod,  every element of $\Proj$ is a direct sum of indecomposable projective modules and the lemma follows.  
\end{proof}

We let $\mt$ be the non-degenerate m-trace on $\Proj$   normalized so that 
$\mt_\Ho(f_\Lambda \circ \epsilon)
 = 1$ where 
$f_\Lambda : \Bbbk \to \Ho$ is the   morphism carrying $1\in \Bbbk$ to  the left cointegral $ \Lambda\in \Ho$.
For $P\in \Proj$,    the map $\Lambda_P': P\to P$ determined by the left action of $\Lambda$ on $P$ is a morphism because $\Lambda$ is both a left and right cointegral.  Consider also  the morphism $\Lambda_P: P\to P$ given by $\Lambda_P=\sum_i x_iy_i$ where $\Omega_P=\sum_i x_i\otimes_\kk y_i $ is defined by Formula~\eqref{E:DefOmegaP}.   
\begin{lemma}\label{L:LP=LP'}
 $\Lambda_P=\Lambda'_P$.  
\end{lemma}
\begin{proof}
First, consider the case when $P=\Ho$.   Since  the space of cointegrals is one dimensional, the space $\Hom_{\Ho\text{-mod}}(\unit, \Ho)$ also is one dimensional and is generated by the morphism  $f_\Lambda : \Bbbk \to \Ho$ defined above.   Combining this with the fact that $\mt$ is non-degenerate (with  the   normalization  above) we get $\Omega_\Ho=f_\Lambda \otimes_\kk \epsilon$.  Thus, for all $x\in \Ho$ we have 
$$
\Lambda_\Ho(x)=f_\Lambda \epsilon (x)=\epsilon(x)\Lambda=x\Lambda=\Lambda'_\Ho(x).
$$
So, the lemma holds for $P=\Ho$.  

Let $P$ be any object in $ \Proj$. Since $\Ho$ is a generator of $\Proj$ in $\Ho$-mod, there is a finite set of pairs of morphisms   $\{f_j:\Ho\to P, g_j:P\to
\Ho\}_{j\in J}$ such that $
  \Id_P=\sum_{j\in J} f_jg_j. 
$
From Proposition \ref{P:alg},  Part \eqref{I:PalgNatural} we get
$$
\Omega_{P}=\sum_{j\in J} \left((f_j \otimes \Id_\unit )\Omega_\Ho (\Id_\unit \otimes g_j) \right)=\sum_{j\in J} \left((f_j \circ f_\Lambda) \otimes (\epsilon \circ g_j) \right).
$$
So by the definition of $\Lambda_P$, for $y\in P$, we have 
\begin{align*}
\Lambda_P(y)=\sum_{j\in J} \left(f_j \circ f_\Lambda \circ \epsilon \circ g_j \right)(y)
&=\sum_{j\in J}f_j(\epsilon(g_j(y))\Lambda)\\
&=\sum_{j\in J}f_j(\Lambda g_j(y))\\
&=\sum_{j\in J}\Lambda f_j(g_j(y))=\Lambda y =\Lambda'_P(y).
\end{align*}
So, $\Lambda_P=\Lambda'_P$.  
\end{proof}

In \cite{DGP} it is shown that the right integral $\lambda \in \Ho^*$ determines a morphism $f_{\lambda,1_\Ho}: \Ho \to \Ho^*\otimes \Ho$  carrying  $1_\Ho$ to $ \lambda \otimes 1_\Ho$.

\begin{lemma}\label{L:Def-d-forHopf}
The morphism $\dd:\Ho\otimes \Ho\to \Ho\otimes \Ho$ defined by
$$\dd=(\rev_\Ho \otimes \Id_{\Ho\otimes \Ho})(\Id_{\Ho}\otimes f_{\lambda,1_\Ho}\otimes \Id_\Ho)(\Id_\Ho\otimes \Delta)$$
 is a chromatic morphism for the generator $\Ho$. 
\end{lemma} 
\begin{proof}
First, we set $h=\lambda(S(\Lambda_{(2)}))\Lambda_{(1)}$ and prove that
\begin{equation}\label{E:lambdaSLam1}
h=\lambda(S(\Lambda_{(2)}))\Lambda_{(1)}=1_\Ho.
\end{equation}
Observe that 
\begin{align*}
S(h)=(\Id_\Ho\otimes \lambda)(S\otimes S)(\Delta(\Lambda))
&=(\Id_\Ho\otimes \lambda)(\Delta^{op}(S(\Lambda)))\\
&=( \lambda\otimes \Id_\Ho)(\Delta(S(\Lambda)))\\
&=\lambda(S(\Lambda))1_\Ho=\lambda(\Lambda)1_\Ho=1_\Ho.
\end{align*}
Since $\Ho$ is finite dimensional, the antipode $S:A\to A $  is invertible and
this computation implies Formula  \eqref{E:lambdaSLam1}.
 
For $x\in X$, let $L_x$ be the left action of $x$ on $\Ho$ or $\Ho^*$.  Note the identity \begin{equation}\label{E:flambdaredblue}
  \lev_\Ho \circ\ (\Id_{\Ho^*} \otimes\ L_x) \circ f_{\lambda \otimes 1_\Ho} = \lambda(x) \cdot \epsilon
\end{equation}
established in the proof of Lemma 3.8 of \cite{DGP}.  

Let $g$ be the morphism on the left side of Formula  \eqref{E:HypothOfdd}.   For any $x, y\in \Ho$,
\begin{align*}
g(x\otimes y)&=(\Id_\Ho\otimes \lev_\Ho)( \Lambda \otimes \Id_\Ho)(\Id_\Ho\otimes f_{\lambda \otimes 1_\Ho} \otimes \Id_\Ho)(x\otimes \Delta(y))\\
&= \Lambda_{(1)}x \otimes \left(\lev_\Ho \big(L_{\Lambda_{(2)}} \otimes \Id_{\Ho} \big)f_{\lambda \otimes 1_\Ho}( y_{(1)})\right)y_{(2)}  \\
&= \Lambda_{(1)}x \otimes \left(\lev_\Ho \big(\Id_{\Ho^*}\otimes L_{S(\Lambda_{(2)})}  \big) f_{\lambda \otimes 1_\Ho}(y_{(1)})\right)y_{(2)}  \\
&= \Lambda_{(1)}x \otimes \lambda(S(\Lambda_{(2)}))  \epsilon(y_{(1)})y_{(2)}  \\
&= \lambda(S(\Lambda_{(2)}))\Lambda_{(1)}x \otimes  y \\
&=x\otimes y
\end{align*}
where the fourth equality follows from Formula  \eqref{E:flambdaredblue}, the fifth equality follows from  the formula $(\epsilon \otimes \Id_\Ho)\Delta=\Id_\Ho$ in the definition of the Hopf algebra~$A$,  and the last equality follows from Formula ~\eqref{E:lambdaSLam1}.  
\end{proof}
Theorem \ref{T:ExHopf} directly follows from the results of this subsection.

\subsection{Proof of Theorem \ref{T:K=KInvlutive}}\label{SS:ProofKuperbergInv}
In Section~\ref{SS:ProofOfThmEXHopf}  we showed that the category  $\Ho\text{-mod}$ satisfies the assumptions of Theorems \ref{T:ExtenFtoRedBlue} and \ref{T:MainExKup}.  Thus we have:
\begin{theorem}
The category of finite dimensional modules over a finite dimensional unibalanced unimodular pivotal Hopf algebra $\Ho$ gives rise to an invariant of bichrome handlebody graphs $F'$ and an   invariant of 3-manifolds $\Kuni_{\Ho\text{-mod}}$.  
 \end{theorem}
 We    show how to  compute   $\Kuni_{\Ho\text{-mod}} (M)$ for a closed connected oriented 3-manifold~$M$.  We will  use the generator $G=\Ho$ and choose $\IP$ to be the ribbon graph formed by the braid closure of the
coupon filled with $h=f_\Lambda \circ \epsilon$, see Section \ref{SS:InvKuni}.   Let $(H_\alpha,\Gamma)$ be a bichrome diagram determined by a Heegaard splitting $M=H_\alpha \cup_{\Sigma} H_\beta$ with lower and upper minimal reducing sets of circles $\{\alpha_1,...,\alpha_g\}$ and $\{\beta_1,...,\beta_g\}$, respectively.   Using the blue graph $\IP$ and the chromatic morphism we change all the red circles $\{\beta_i\}$ into a blue graph.   Now $\{\alpha_i\}$ is a minimal reducing set of circles on $H_\alpha$ and so using Property (4) of Theorem~\ref{P:F'onProjHandlebodies} we can cut along the discs in $H_\alpha$ bounded by these circles to obtain a graph $\Gamma'$ in the boundary of the 3-ball $B^3$.  By definition, 
 \begin{equation}\label{E:Kuni=B3Gam'}
 \Kuni_{\Ho\text{-mod}}(M)=F'(B^3,\Gamma'), 
 \end{equation}  see Figure \ref{F:Computation} for an example.   
 \begin{figure}[htbp]
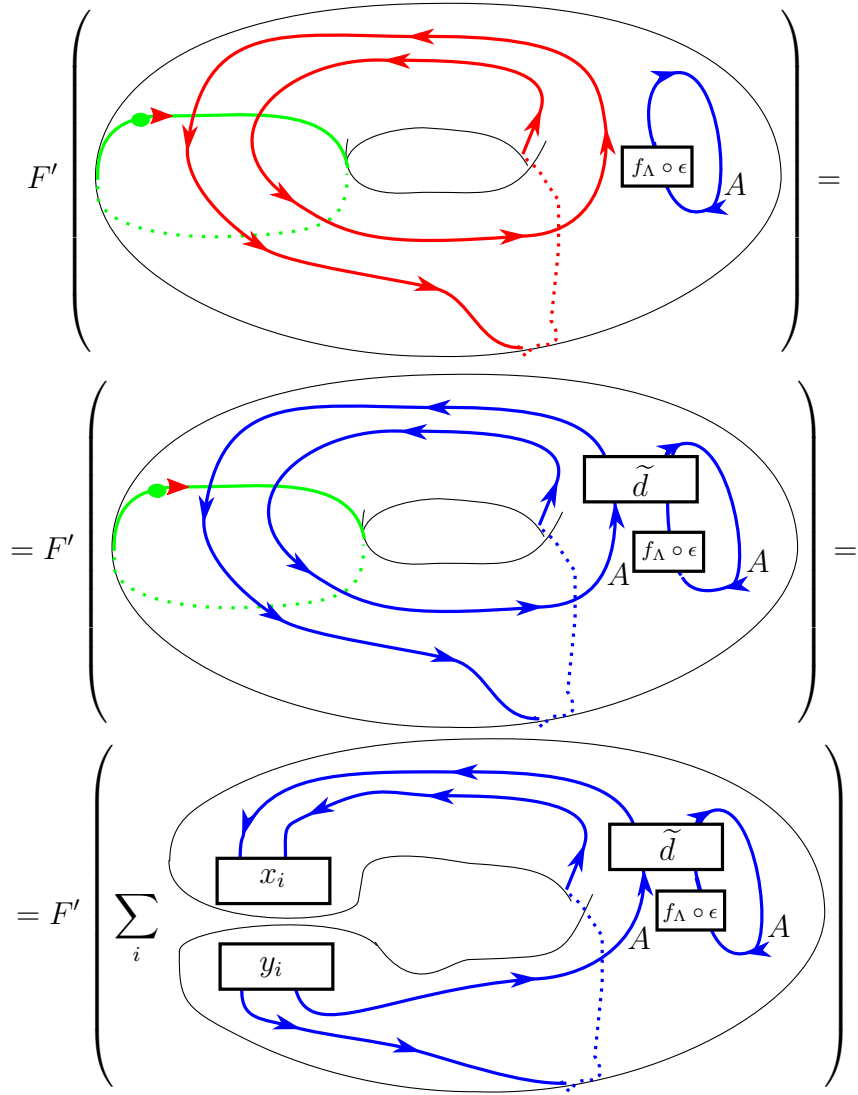

\begin{center}
$$
F'\left(\epsh{fig21-francois}{26ex}\put(-18,-2){\Ho}\put(-47,5){\tiny{$f_\Lambda\circ \epsilon$}}\right)=$$
$$
=F'\left(\epsh{fig22-francois}{26ex}\put(-16,-4){{\Ho}}\put(-49,1){\tiny{$f_\Lambda\circ \epsilon$}}\put(-52,20){$\dd$}\put(-60,-8){\Ho}\right)=
$$
$$
=F'\left(\sum_i\epsh{fig23-francois}{26ex}\put(-18,-4){\Ho}\put(-51,2){\tiny{$f_\Lambda\circ \epsilon$}}\put(-52,20){$\dd$}\put(-62,-8){{\Ho}}\put(-178,13){$x_i$}\put(-178,-15){$y_i$}\right)
$$

\caption{Computation of   $\Kuni_{\Ho\text{-mod}}$ for the lens space $L(2,1)$. The first equality is obtained  using the chromatic morphism to transform the red circle
into a blue graph. Then cutting along the meridian (depicted on the left) we obtain  a graph $\Gamma'\subset \partial B^3$.  By definition,  $\Kuni_{\Ho \text{-mod}}(L(2,1))=F'(B^3;\Gamma')$. }
\label{F:Computation}
\end{center}
\end{figure}

In the rest of the section, we assume  $\Ho$ to be  involutive so that $S^2={\rm {id}}_A$ and compute $\Kuni_{\Ho\text{-mod}}(M)$.  The involutivity of $\Ho$ implies that we can choose the pivotal structure of $\Ho$-mod to be trivial and then there is a forgetful pivotal functor from $\Ho$-mod to the category of vector spaces $\Vec$.   
We now state  two lemmas.  The proof of the first is straightforward and we leave it to the reader.  

\begin{lemma}\label{L:ForgetFunctorRT}
  We have the following commutative diagram:
$$\begin{tikzpicture}
\matrix(m)[matrix of math nodes,
row sep=2.6em, column sep=4em,
text height=1.5ex, text depth=0.25ex]
{Rib_{\Ho\text{-mod}} & Rib_{\Vec}\\
\Ho\text{-mod} & \Vec\\};
\path[->,font=\scriptsize,>=angle 90]
(m-1-1) edge node[auto] {Forgetful} (m-1-2)
        edge node[auto] {$F$} (m-2-1)
(m-1-2) edge node[auto] {$F_\Vec$} (m-2-2)
(m-2-1) edge node[auto] {Forgetful} (m-2-2);
\end{tikzpicture}
$$
where the horizontal arrows represent the forgetful functors and the vertical arrows represent  the  pivotal Penrose functors from the categories of  planar colored ribbon graphs   to $\Ho$-mod and  $\Vec$ respectively.  
\end{lemma}

The next lemma says that evaluating red circles with the chromatic morphism is essentially the integral. 
This lemma directly  follows from the definition of the chromatic morphism and Formula  \eqref{E:flambdaredblue}.  

\begin{lemma}\label{lem:integralasd}
For all $x\in \Ho$ the following equality holds in $\Vec$:
$$
(\lev_{\Ho}\otimes \Id_\Ho)(\Id_{\Ho^*} \otimes L_x\otimes \Id_\Ho)(\Id_{\Ho^*}\otimes \dd)(\rcoev_\Ho \otimes \Id_\Ho)=\lambda(x) \Id_\Ho .
$$
This equality   is represented pictorially as  
 $$
\epsh{fig29-francois}{12ex}\put(-20,3){{\small $\wt d$}}\put(-26,22){$x$}=\lambda(x)\epsh{fig30-francois}{12ex}
$$
where all   strands are colored with  $\Ho$ and the thick dot   is labeled with $x$.
\end{lemma}

Now we can  compute $\Kuni_{\Ho\text{-mod}}(M)$  (for involutive   $\Ho$) and prove Theorem \ref{T:K=KInvlutive}.  By Formula  \eqref{E:Kuni=B3Gam'}, we need to compute  $F'(B^3,\Gamma')$.  Let $T_\Ho$ be a (1,1)-ribbon graph whose closure is $\Gamma'$.    Formula  \eqref{E:DefF'} implies that 
$$
F'(B^3,\Gamma')=\mt_\Ho(F(T_\Ho)).
$$ 
Recall the  $\Ho$-mod morphism  $f_\Lambda \circ \epsilon$ defined at the beginning of Section~\ref{SS:ProofOfThmEXHopf}.  Since $\mt_\Ho(f_\Lambda \circ \epsilon)=1$, Theorem \ref{T:K=KInvlutive} follows from the next claim:
   \begin{claim}\label{C:FT_Ho=Kup} $F(T_\Ho)= \Kup_\Ho(M) \cdot (f_\Lambda \circ \epsilon).$ 
   \end{claim}
 \begin{proof}[Proof of Claim \ref{C:FT_Ho=Kup}]
 Lemma \ref{L:ForgetFunctorRT} implies that $$\Forgetful(F(T_\Ho))=F_{\Vec}(\Forgetful(T_\Ho)).$$ 
 Thus, it suffices to prove that 
   \begin{equation}\label{E:FVec=Kup}
   F_{\Vec}(\Forgetful(T_\Ho))= \Kup_\Ho(M) \cdot \Forgetful (f_\Lambda \circ \epsilon)
   \end{equation}
    since $\Forgetful$ is a faithful functor.   In the rest of the proof, we   work in the category  $\Vec$. This allows  us to consider the left multiplication by an element of $ \Ho$ as a morphism in $\Vec$.  To simplify notation we identify each morphism in $\Ho-$mod with its underlying linear map.  

 \begin{figure}[htbp]
\begin{center}
$$
\Forgetful\left(F\left(\sum_i \epsh{fig24-francois}{26ex}\put(-38,-12){\Ho}\put(-47,4){\tiny{$f_\Lambda\circ \epsilon$}}\put(-52,20){$\dd$}\put(-62,-8){{\Ho}}\put(-178,13){$x_i$}\put(-178,-15){$y_i$}\right)\right)=$$
$$=F_{\Vec}\left(\sum_i\epsh{fig25-francois}{26ex}\put(-26,-12){{\Ho}}\put(-38,7){\tiny{$f_\Lambda\circ \epsilon$}}\put(-57,-10){\Ho}\put(-42,25){$\dd$}\put(-178,13){$x_i$}\put(-178,-15){$y_i$}\right)=$$
$$
=F_{\Vec}\left(\hspace{25pt}\epsh{fig26-francois}{26ex}\put(-30,-10){\Ho}\put(-41,7){\tiny{$f_\Lambda\circ \epsilon$}}\put(-50,25){$\dd$}\put(-55,-10){{\Ho}}\put(-172,13){$S(\Lambda_{(2)})$}\put(-222,15){$S(\Lambda_{(1)})$}\right)=
$$
$$
=F_{\Vec}\left(\epsh{fig27-francois}{15ex}\put(-28,-28){\Ho}\put(-60,30){{\tiny $S(\Lambda_{(2)})S(\Lambda_{(1)})$}}\put(-40,-13){{\tiny{$f_\Lambda\circ \epsilon$}}}\put(-45,3){$\dd$}\put(-55,-30){{\Ho}}\right)=\lambda\left(S(\Lambda_{(2)})S(\Lambda_{(1)})\right)\cdot(f_\Lambda\circ \epsilon)
$$

\caption{In Figure \ref{F:Computation} we showed $\Kuni_{\Ho\text{-mod}}(L(2,1))=F'(B^3,\Gamma')$; here we continue this computations when $\Ho$ is involutive.
The first drawing depicts the 1-1 tangle obtained by cutting $\Gamma'$ along the black disc.  In the first equality, we pass to the category $\Vec$ using Lemma \ref{L:ForgetFunctorRT}. Since $\Vec$ has a symmetric braiding, the crossings make sense. The third equality re-expresses the cutting through beads (see Lemma \ref{L:LP=LP'}).
Finally, we collect the beads and apply Lemma \ref{lem:integralasd}. }
\label{F:ComputationVec}
\end{center}
\end{figure}    
Recall that  cutting at one of the $\alpha$-circles,  we replace the blue circles crossing it  
 with an ordered   pair of coupons filled with a sum $\sum_i x_i\otimes_\kk y_i$, see  Formula  \eqref{E:OmegaDefMuilt}.   By  Lemma \ref{L:LP=LP'} the corresponding  morphism is   the left action by $\Lambda$.  We provide  each  circle $\alpha_i$ with an orientation and a base point.   As above, we  use the chromatic morphism to change all   red circles $\{\beta_i\}$ into blue graphs (this yields  an orientation on each   $\beta_i$).  Instead of cutting the Heegaard surface along the   circles $\{\alpha_i\}$ we decorate the blue graph with certain  elements of $\Ho$ called beads. Namely, each circle $\alpha_j$ intersects the   upper circles $\{\beta_i\}$ transversely, and we let $c_1, . . . , c_m$ be these intersections enumerated  in the order they are met when     traveling  along $\alpha_j$ in the positive direction starting and ending in  the base point.  For  $k=1,..., m$, set $p_k=0$ if the tangent vectors of  $\alpha_j$ and the  circle $\beta_i$ meeting $\alpha_j$ at $c_k$  (taken in this order)  form a positively oriented basis in the tangent space at $c_k$; otherwise set $p_k=1$.
We assign to each $c_k$ the \lq\lq bead'' $S^{p_k}(\Lambda_{(k)})$ where 
$\Delta^m(\Lambda)=\Lambda_{(1)}\otimes...\otimes  \Lambda_{(m)}$.

Since $\Vec$ has a trivial ribbon structure,   the left hand side of Formula  \eqref{E:FVec=Kup} depends  only  on 
the abstract graph of $T_\Ho$ or, equivalently, on $\Gamma'$.  
 Thus, we can compute $F_{\Vec}(\Forgetful(T_\Ho))$ as follows.  Each upper circle $\beta_i$ has an orientation and a base point (determined by where the chromatic morphism is applied).  Starting at this point and following the orientation we collect the beads to obtain a word $a_i$ of $\Ho$ written from right to left. Doing this for all upper circles we obtain $g$ beads:  $a_1,...,a_g$.  
For each $\beta_i$, apply Lemma \ref{lem:integralasd} where $x$ is the bead $a_i$ to obtain 
$$F_{\Vec}(\Forgetful(T_\Ho))=\lambda(a_1)\lambda(a_2)...\lambda(a_g)\cdot (f_\Lambda \circ \epsilon).$$
For an example of this computation see Figure \ref{F:ComputationVec}.  
 But since $\Ho$ is involutive,  the definition of the Kuperberg invariant implies  that   
 $$
\Kup_\Ho(M)= \lambda(a_1)\lambda(a_2)...\lambda(a_g).
 $$   
This implies our claim above and  completes the proof of the theorem.  
 \end{proof}

\section{Proof of Theorem \ref{T:Kuni=TV}} \label{S:ProofTVinv}

In this section we will   use  notation of Section \ref{SS:TVinvIntro}.   
  Let $\mathcal T$ be a triangulation of a closed connected oriented $3$-manifold $M$   and let $t$ be a maximal tree of edges of $\mathcal T$. Clearly,  $t$ contains all the vertices of $\mathcal T$.
Let $H_\beta$ be a regular neighborhood of the $1$-skeleton of $\mathcal T$ in~$M$ and let $H_\alpha=M \setminus ({\rm {Int}} (H_\beta)$. Then $M=H_\alpha \cup H_\beta$ is a Heegaard splitting of~$M$.  The $\beta$-circles are meridians of the edges of $\mathcal T$ not in the tree $t$ and the $\alpha$-circles bound discs lying in the 2-dimensional faces of $\mathcal T$.
  As in Section \ref{SS:HeegSplitting} this Heegaard diagram gives a bichrome handlebody graph $(H_\alpha,\Gamma)$ where the red graph consists of  the $\beta$-circles.  

 By the definition of the chromatic morphism, the value of a red unknot
 is
 the  dimension $\mathcal D \neq 0 $ of $\cat$ (see Formula  \eqref{eq:dimC}).
Thus, the value of $F'(H_\alpha,\Gamma)$ does not change if we multiply it by $1/\mathcal D $ and at the same time add a red unknot to the bichrome handlebody graph $(H_\alpha,\Gamma)$.  With this in mind, we construct a new bichrome handlebody graph as follows.  Starting with $(H_\alpha,\Gamma)$ place a red unknot on the boundary of the neighborhood of each edge of the tree $t$.   Let $e$ be a leaf of $t$, i.e.\ an edge of $t$ such that one of its vertices is not adjacent to other edges of~$t$.  Then we can slide the new red unknot associated to $e$ over   the red meridians of all other edges adjacent to this vertex.  The resulting unknot is  a red meridian around $e$.  Continuing this process on the leaves of $t\setminus \{e\}$ etc., we obtain a bichrome handlebody graph $(H_\alpha,\Gamma')$  such that  each edge of $\mathcal T$ has a red meridian going around this edge.  Using this construction and  the sliding property of $F'$   for red circles, we get 
\begin{equation}\label{E:KuniHGG'}
\Kuni_\cat(M)=F'(H_\alpha,\Gamma)=\frac{1}{\mathcal D ^{v-1}}F'(H_\alpha,\Gamma')
\end{equation}
where $v$ is the number of vertices of $\mathcal T$.  

Now we compute $F'(H_\alpha,\Gamma')$.  Use the chromatic morphism to make all the red circles blue.  By definition of the chromatic morphism, each red circle is changed to a blue circle colored with $G$ or, equivalently, colored with the weighted sum $\sum_i \qdim(S_i) S_i$ (note this happens with  a meridian of each edge of $\mathcal T$).  Furthermore,  cutting along the discs formed by the 2-dimensional faces of $\mathcal T$ we obtain a set of
spherical
tetrahedra (indexed by the set of tetrahedra of $\mathcal T$)  whose edges are all colored by $G$ and whose four 3-legged coupons are filled with morphisms coming from the cutting.  Since $G$ splits as an orthogonal direct sum of simple objects, each of these
spherical
 tetrahedra is
 a
 sum indexed by colorings of the edges of $T$ by elements of the set $\{S_i\}$.  Moreover, each component of this sum is 
proportional to the corresponding   $6j$-symbol.  Therefore,  
$$ F'(H_\alpha,\Gamma')=  {\mathcal D ^{v}}\,\,
\TV(M)$$
where $\TV(M) $ is the Turaev-Viro invariant of~$M$ computed  on  the triangulation $\mathcal T$ of $M$, see \cite{Tu}.  Combining this with Formula  \eqref{E:KuniHGG'}, we obtain that  $ {\mathcal D }^{-1} \Kuni_\cat(M)=\TV(M)$.

\linespread{1}

\end{document}